\documentclass{myelsarticle}
\usepackage{amsmath}
\usepackage{graphicx}
\usepackage{amssymb}
\usepackage{epstopdf}
\usepackage{color}
\usepackage{tikz}
\usetikzlibrary{calc}
\usepackage{subcaption}
\usepackage{enumitem}
\usepackage{todonotes}
\usepackage{hyperref}
\usepackage{bbm}
\usepackage{soul}
\usepackage{float}

\newcommand{\RR}{\mathbb{R}}
\newcommand{\conv}{\operatorname{conv}}

\newcommand{\unitvec}{\mathbbm{e}}

\newtheorem{thm}{Theorem}[section]
\newtheorem{lem}{Lemma}

\newproof{proof}{Proof}

\begin{document}
    \title{Hidden Vertices in Extensions of Polytopes}

    \author{Kanstantsin Pashkovich}
    \address{Universit\'{e} libre de Bruxelles, D\'{e}partement de Math\'{e}matique, Boulevard du
    Triomphe, B-1050 Brussels, Belgium}
    \ead{kanstantsin.pashkovich@gmail.com}

    \author{Stefan Weltge}
    \address{Otto-von-Guericke-Universit\"at Magdeburg, Universit\"atsplatz 2, D-39106 Magdeburg, Germany}
    \ead{weltge@ovgu.de}

    \begin{abstract}
        Some widely known compact extended formulations have the property that each vertex of the corresponding extension
        polytope is projected onto a vertex of the target polytope. In this paper, we prove that for heptagons with vertices in general position none of the minimum size extensions has this property. 
Additionally, for any $ d \geq 2 $ we construct a family of $
        d $-polytopes such that at least $ \frac{1}{9} $ of all vertices of any of their minimum size extensions is not projected onto vertices.
    \end{abstract}

    \begin{keyword}
        extended formulation \sep polytope \sep minimal extension \sep projection \sep additional vertex
    \end{keyword}

    \maketitle

    \makeatletter{}\section{Introduction}

The theory of \emph{extended formulations} is a fast developing research field that has connections to many other fields
of mathematics. In its core, it deals with the concept of representing polytopes (usually having many facets or even no
known linear description) as linear projections of other polytopes (which, preferably, permit smaller linear
descriptions). Thus, polyhedral theory plays a crucial role for extended formulations establishing a natural connection
to geometry.

Recall that a \emph{polytope} is the convex hull of a finite set of points and that every polytope can be described as
the solution set of a system of finitely many linear inequalities and equations. The \emph{size} of a polytope is the
number of its facets, i.e., the minimum number of inequalities in a linear description of the polytope. An alternative
way to represent a polytope is to write it as a projection of another polytope. Concretely,
a polytope~$Q\subseteq \RR^{n}$ is called an \emph{extension} of a polytope~$P \subseteq \RR^{d}$ if the orthogonal
projection of $Q$ on the first $ d $ coordinates equals $ P $. The \emph{extension
complexity} of a polytope~$P$ is the minimum size of any extension for $P$. Here, we restrict extensions to be
polytopes, not polyhedra, as well as projections to be orthogonal, not general affine maps. This definition simplifies the representation, however does not lead to loss of generality.

As an illustration, regular hexagons (having six facets) can be written as a projection of triangular prisms (having
only five facets), see Fig.~\ref{fig:example}. It is easy to argue that such extensions are indeed of minimum size.
Another textbook example is the $d$-dimensional \emph{cross polytope}, which is the convex hull of all unit vectors $
\unitvec_i \in \RR^d $ and their negatives $ -\unitvec_i $ for $ i =1,\dotsc,d $. While the $d$-dimensional cross
polytope has $ 2^d $ facets, it can be written as the projection onto the $x$-coordinates of the polytope
\[
    \bigg\{ (x,y) \in \RR^d \times \RR^{2d} \, : \,
        x = \sum_{i=1}^d \lambda_i \unitvec_i - \sum_{i=1}^d \lambda_{d+i} \unitvec_i, \,
        \sum_{j=1}^{2d} \lambda_j = 1, \,
        \lambda_j \geq 0 \ \forall \, j=1,\dotsc,2d \bigg\},
\]
which is a $(2d-1)$-simplex (and hence has only $2d$ facets) and can also be proven to be of minimum size. Note that both
examples admit minimum size extensions whose vertices are again projected to vertices.

\begin{figure}
    \begin{center}
    \makeatletter{}\begin{tikzpicture}[scale=0.8]

        \def \cpolytop {black}
    \def \cschatten {black}
    \def \cschattenlinien {black}

    \def \oebene {0.1}
    \def \oebenelinien {0.3}
    \def \opolytop {0.4}
    \def \olinienpolytophinten {0.5}
    \def \oschatten {0.5}
    \def \oschattenlinien {0.7}

        \coordinate (O) at (0,0);

        \def \alpha {23}     \def \beta {-52} 
        \def \zoom {1.0}

        \pgfmathsetmacro{\xx}{sin(\beta)*\zoom}
    \pgfmathsetmacro{\xy}{-sin(\alpha)*cos(\beta)*\zoom}
    \pgfmathsetmacro{\yx}{cos(\beta)*\zoom}
    \pgfmathsetmacro{\yy}{sin(\alpha)*sin(\beta)*\zoom}
    \pgfmathsetmacro{\zx}{0}
    \pgfmathsetmacro{\zy}{cos(\alpha)*\zoom}

        \coordinate (x) at (\xx,\xy);
    \coordinate (y) at (\yx,\yy);
    \coordinate (z) at (\zx,\zy);

        \draw[fill=black,opacity=\oebene] (O) -- ($ 6*(x) $) -- ($ 6*(x) + 6*(y) $) -- ($ 6*(y) $) -- cycle;
    \foreach \i in {0,...,6}
        \draw[color=black,opacity=\oebenelinien] ($ (O) + \i*(y) $) -- ($ (O) + \i*(y) + 6*(x) $);
    \foreach \i in {0,...,6}
        \draw[color=black,opacity=\oebenelinien] ($ (O) + \i*(x) $) -- ($ (O) + \i*(x) + 6*(y) $);

        \coordinate (u1) at ($ (O) + 4*(x) + 2*(y) + 2*(z) $);
    \coordinate (u2) at ($ (O) + 4*(x) + 4*(y) + 2*(z) $);
    \coordinate (u3) at ($ (O) + 2*(x) + 4*(y) + 2*(z) $);
    \coordinate (u4) at ($ (O) + 2*(x) + 2*(y) + 2*(z) $);
    \coordinate (o1) at ($ (O) + 5*(x) + 3*(y) + 3.5*(z) $);
    \coordinate (o2) at ($ (O) + 1*(x) + 3*(y) + 3.5*(z) $);

        \coordinate (s1) at ($ (O) + 5*(x) + 3*(y) $);
    \coordinate (s2) at ($ (O) + 4*(x) + 4*(y) $);
    \coordinate (s3) at ($ (O) + 2*(x) + 4*(y) $);
    \coordinate (s4) at ($ (O) + 1*(x) + 3*(y) $);
    \coordinate (s5) at ($ (O) + 2*(x) + 2*(y) $);
    \coordinate (s6) at ($ (O) + 4*(x) + 2*(y) $);

        \draw[opacity=\olinienpolytophinten] (u1) -- (u4) -- (u3);
    \draw[opacity=\olinienpolytophinten] (u4) -- (o2);

        \fill[color=\cpolytop,opacity=\opolytop] (o1) -- (u1) -- (u2) -- (u3) -- (o2) -- cycle;

        \draw[thick] (o1) -- (u1) -- (u2) -- (u3) -- (o2) -- cycle;
    \draw[thick] (o1) -- (u2);

        \fill[color=\cschatten,opacity=\oschatten] (s1) -- (s2) -- (s3) -- (s4) -- (s5) -- (s6) -- cycle;

\end{tikzpicture}
 
    \end{center}
    \caption{A hexagon (shadow) as a projection of a triangular prism.}
    \label{fig:example}
\end{figure}
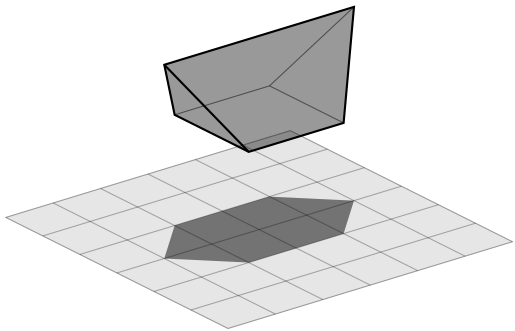

In fact, many widely known extended formulations have the property that every vertex of the corresponding
extension is projected onto a vertex of the target polytope. See, for instance, extended formulations for the
parity polytope~\cite{Yannakakis91,CarrKonjevod2005},
the permutahedron (as a projection of the Birkhoff-polytope),
the cardinality indicating polytope~\cite{Pashkovich},
orbitopes \cite{FaenzaK09},
or spanning tree polytopes of planar graphs~\cite{Williams2002}.
Although there are not many polytopes whose extension complexity is known exactly, most of the mentioned extensions have
minimum size at least up to a constant factor.
Moreover, for many of these extensions there is even a one-to-one correspondence between the vertices of the extension
and the vertices of the target polytope.
Clearly, a general
extension might have vertices that are not projected onto vertices. Here, let us call such vertices to be \emph{hidden
vertices}. The following natural question arises: Given a polytope $P$, can we always find a minimum size extension $Q$
of $ P $ that has no hidden vertices?

In this paper, we negatively answer the above question. Namely, we prove that for almost all heptagons, every minimal
extension has at least one hidden vertex. Later we extend this result and construct a family of $d $-polytopes, $ d \geq
2 $, such that  at least $ \frac{1}{9} $ of all vertices in any minimum size extension are hidden.

Thus,  in this paper we show that there are polytopes for which the minimum size of an extension without hidden vertices is strictly bigger than their extension complexity. Consider the open question: How big can be the difference between the minimum size of extensions without hidden vertices and the extension complexity? This paper demonstrates that the difference can be at least one in the context of the above question. However, the above question remains open and as the next step to study hidden vertices we propose the following question: Is there a polynomial $q:\RR\rightarrow \RR$ such that for every polytope the minimum size of its extension without hidden vertices is at most $q(s)$, where $s$ is the extension complexity of the polytope?
 
    \makeatletter{}\section{Minimum Extensions of Heptagons}

In this section, we consider convex polygons with seven vertices taken in general position. For such polygons, we prove
that there is no extension of minimum size such that every vertex of the
extension is projected onto a vertex of the polygon.

\subsection{Extension Complexity of Heptagons}
Let us briefly recall known facts about extensions of heptagons. In 2013, Shitov \cite{Shitov14} showed that the
extension complexity of any convex heptagon is at most $ 6 $. Further, it is easy to see that any affine image of a
polyhedron with only $ 5 $ facets has at most $ 6 $ vertices. Thus, one obtains:

\begin{thm}[Shitov \cite{Shitov14}]
    \label{thm:extension_complexity_heptagon}
    The extension complexity of any convex heptagon is $ 6 $.
\end{thm}

\noindent
While Shitov's proof is purely algebraic, independently, Padrol and Pfeifle \cite{PadrolP14} established a geometric proof
of this fact. In fact, they showed that any convex heptagon can be written as the projection of a $ 3 $-dimensional
polytope with $ 6 $ facets. In order to get an idea of such a polytope, let us consider the following construction
(which is a dual interpretation of the ideas of Padrol and Pfeifle):

Let $ P $ be a convex heptagon with vertices $ v_1,\dotsc,v_7 $ in cyclic order. For $ i \in \{ 2,3,5,6,7 \} $ let us set $ w_i :=
(v_i,0) \in \RR^3 $. Further, choose some numbers $ z_1,z_4 > 0 $ such that $ w_1 := (v_1,z_1) $, $ w_4 := (v_4,z_4)
$, $ w_2 $ and $ w_3 $ are contained in one hyperplane and consider $ Q' := \conv(\{w_1,\dotsc,w_7\}) $. It can be
shown \cite{PadrolP14} that (by possibly shifting the vertices' labels) one may assume that the convex hull of $ w_1 $,
$ w_4 $ and $ w_6 $ forms a facet $ F $ of $ Q' $. In this case, remove the defining inequality of $ F $ from an
irredundant outer description of $ Q' $ and obtain a $ 3 $-dimensional polytope $ Q $ with only $ 6 $ facets whose projection
is still $ P $. For an illustration, see Fig. \ref{fig:extension_heptagon}. Note that removing the facet $ F $ results
in an additional vertex that projects into the interior of $ P $.

In what follows, our argumentation does not rely on the construction described above but only on the statement of
Theorem \ref{thm:extension_complexity_heptagon}. However, the previous paragraph gives an intuition why additional
vertices may help in order to reduce the number of facets of an extension.

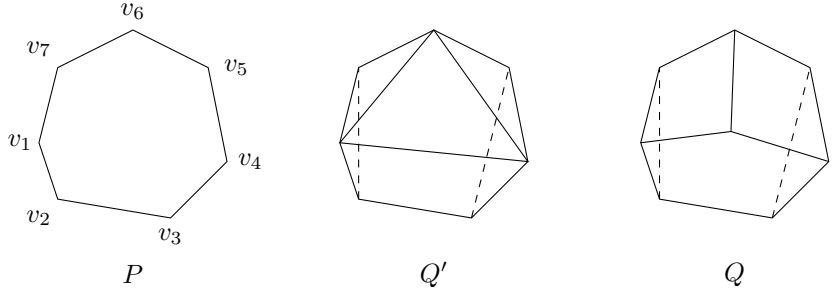
\begin{figure}
    \centering
    \makeatletter{}\begin{tikzpicture}[scale=0.25]
        \coordinate (p1) at (1,5);
    \coordinate (p2) at (2,2);
    \coordinate (p3) at (8,1);
    \coordinate (p4) at (11,4);
    \coordinate (p5) at (10,9);
    \coordinate (p6) at (6,11);
    \coordinate (p7) at (2,9);

    \draw (p1) -- (p2) -- (p3) -- (p4) -- (p5) -- (p6) -- (p7) -- cycle;
    \node at (6,-2) { $ P $ };
    \node at (0,5) { $ v_1 $ };
    \node at (1,1) { $ v_2 $ };
    \node at (8,0) { $ v_3 $ };
    \node at (12.25,4) { $ v_4 $ };
    \node at (11.5,9) { $ v_5 $ };
    \node at (6,12) { $ v_6 $ };
    \node at (1,10) { $ v_7 $ };

        \coordinate (q1) at (17,5);
    \coordinate (q2) at (18,2);
    \coordinate (q3) at (24,1);
    \coordinate (q4) at (27,4);
    \coordinate (q5) at (26,9);
    \coordinate (q6) at (22,11);
    \coordinate (q7) at (18,9);

    \draw[dashed] (q2) -- (q7);
    \draw[dashed] (q3) -- (q5);
    \draw (q1) -- (q2) -- (q3) -- (q4) -- (q5) -- (q6) -- (q7) -- cycle;
    \draw (q1) -- (q4) -- (q6) -- cycle;
    \node at (22,-2) { $ Q' $};

        \coordinate (f1) at (33,5);
    \coordinate (f2) at (34,2);
    \coordinate (f3) at (40,1);
    \coordinate (f4) at (43,4);
    \coordinate (f5) at (42,9);
    \coordinate (f6) at (38,11);
    \coordinate (f7) at (34,9);
    \coordinate (m) at (37.8,5.6);

    \draw[dashed] (f2) -- (f7);
    \draw[dashed] (f3) -- (f5);
    \draw (f1) -- (f2) -- (f3) -- (f4) -- (f5) -- (f6) -- (f7) -- cycle;
    \draw (f1) -- (m) -- (f4);
    \draw (m) -- (f6);
    \node at (38,-2) { $ Q $};
\end{tikzpicture}
 
    \caption{Example of the construction of a $ 3 $-dimensional extension $ Q $ with $ 6 $ facets for a heptagon $ P $.}
    \label{fig:extension_heptagon}
\end{figure}

\subsection{Additional Vertices of Minimum Size Extensions of Heptagons}

In this section, we will show that most convex heptagons force minimum size extensions to have at least one vertex that is
not projected onto a vertex. In order to avoid singular cases in which it is possible to construct minimum size extensions
without additional vertices, we only consider convex heptagons $ P $ that satisfy the following three conditions:

\begin{enumerate}
    \item \label{cond:parallel} There are no four pairwise distinct vertices $ u_1, \dotsc, u_4 $ of $ P $ such that the
    lines $ \overline{u_1 u_2} $, $ \overline{u_3 u_4} $ are parallel.
    \item \label{cond:threeintersect} There are no six pairwise distinct vertices $ u_1, \dotsc, u_6 $ of $ P $, such that the
    lines $ \overline{u_1 u_2} $, $ \overline{u_3 u_4} $, $ \overline{u_5 u_6} $ have a point common to all three of them.
    \item \label{cond:special} There are no seven pairwise distinct vertices $u_1, \dotsc, u_7 $ of $ P $ such that the
    intersection points $ \overline{u_1 u_2} \cap \overline{u_3 u_4} $, $ \overline{u_2 u_5} \cap \overline{u_4 u_6} $ and $
    \overline{u_3 u_7} \cap \overline{u_1 u_5} $ lie in the same line.
\end{enumerate}

\noindent Here, a convex heptagon $ P $ is called to be \emph{in general position}, if it satisfies conditions
(\ref{cond:parallel})--(\ref{cond:special}). We are now ready to state our main result:

\begin{thm}
    \label{thm:additional_vertices_heptagon}
    Let $ P $ be a convex heptagon in general position. Then any minimum size extension of $ P $ has a vertex that is not
    projected onto a vertex of $ P $.
\end{thm}

\noindent
From now on, let us fix a convex heptagon $ P $ that is in general position. In order to prove Theorem
\ref{thm:additional_vertices_heptagon}, let us assume, for the sake of contradiction, that there exists a polytope $ Q $ with
only six facets such that $ Q $ is an extension of $ P $ and every vertex of $ Q $ is projected onto a vertex of $ P $.
Towards this end, let us first formulate two Lemmas, which we will extensively use through the whole consideration.

\begin{lem}
    \label{lem:three_points}
    Let $ w_1,\dotsc,w_4 $ be four pairwise distinct vertices of $ Q $ such that exactly one pair of them is projected onto the
    same vertex of $ P $. Then, the dimension of the affine space generated by $w_1$, \ldots ,$w_4$  equals $3$.
\end{lem}
\begin{proof}
    Let us assume the contrary and let $ w_1,\dotsc,w_4 $ be such vertices of $Q$ that the dimension of the corresponding affine space is at most $2$. Then, the dimension of the affine space generated
    by the projections of $ w_1,\dotsc,w_4 $ is at most one since it is the projection of the affine space generated by $ w_1,\dotsc,w_4 $, while two distinct points in this space are projected onto the same point. This implies that the projections of  $w_1$, \ldots, $w_4$, and thus three different vertices of $ P $, lie on the same line, a contradiction.
\end{proof}

\begin{lem}
    \label{lem:no_prism}
    There are no six vertices $ w_1,\dotsc,w_6 $ of $ Q $ such that
    \begin{itemize}
        \item[--] at most one pair of them is projected onto the same point, and
        \item[--] $ \conv(\{w_1,\dotsc,w_6\}) $ is a triangular prism.
    \end{itemize}
\end{lem}
\begin{proof}
    Let $ w_1,\dotsc,w_6 $ be any six vertices of $ Q $ that form a triangular prism $ \Gamma $.
    We claim that two of them are projected onto the same point.
    Otherwise, label the vertices of $ \Gamma $ as in Fig.~\ref{fig:prism}.
    \begin{figure}
        \centering
        \makeatletter{}\begin{tikzpicture}[scale=0.6]
    \coordinate (w1) at (-2,-2.5);
    \coordinate (w2) at (2,-2);
    \coordinate (w3) at (-1,-1);
    \coordinate (w4) at (1,-1);
    \coordinate (w5) at (-2,1);
    \coordinate (w6) at (2,0);
    \draw (w1)--(w2)--(w4)--(w3)--(w5)--(w6)--(w2);
    \draw (w4)--(w6);
    \draw (w3)--(w1)--(w5);
    \node[below, left] at (w1) {$w_1$};
    \node[left] at (w3) {$w_3$};
    \node[above, left] at (w5) {$w_5$};
    \node[below, right] at (w2) {$w_2$};
    \node[right] at (w4) {$w_4$};
    \node[above, right] at (w6) {$w_6$};
\end{tikzpicture}
 
        \caption{Labeling of the vertices of $ \Gamma $.}
        \label{fig:prism}
    \end{figure}
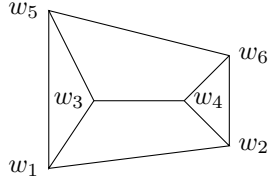
    The lines $ \overline{w_1 w_2} $ and $ \overline{w_3 w_4} $ are not skew, since the points $ w_1 $, $ w_2 $, $ w_3
    $, $ w_4 $ lie in the same facet of $\Gamma$.
    By condition~(\ref{cond:parallel}), the lines $ \overline{w_1 w_2} $ and $ \overline{w_3 w_4} $ are not parallel,
    since otherwise the lines $\overline{u_1 u_2}$ and $\overline{u_3 u_4}$ are parallel, where $ u_i $ denotes the
    projection of $ w_i $ for $ i = 1,\ldots,6 $.
    Thus, the lines $ \overline{w_1 w_2} $ and $ \overline{w_3 w_4} $ have a unique common point.
    Analogously, the lines $ \overline{w_3 w_4} $ and $ \overline{w_5 w_6} $ have a unique common point.
    Note, that the points $ \overline{w_1 w_2} \cap \overline{w_3 w_4} $ and $ \overline{w_3 w_4} \cap \overline{w_5
    w_6} $ lie in the hyperplane corresponding to the facet of $ \Gamma $ containing $ w_1 $, $ w_2 $, $ w_5 $ and $ w_6
    $, since the lines $ \overline{w_1 w_2} $ and $ \overline{w_5 w_6} $ lie in this hyperplane.
    Since the line $ \overline{w_3 w_4} $ is not contained in this hyperplane, it has at most one common point with this
    hyperplane, showing that $ \overline{w_1 w_2} \cap \overline{w_3 w_4} = \overline{w_3 w_4} \cap \overline{w_5 w_6}
    $.
    Hence, the lines $ \overline{u_1 u_2} $, $ \overline{u_3 u_4} $ and $ \overline{u_5 u_6} $ have a point common to
    all three of them, which contradicts condition~(\ref{cond:threeintersect}).

    Suppose now that exactly one pair of vertices of~$ \Gamma $ is projected onto the same point.
    Let us denote this point by~$ u $.
    Since $ u $ is a vertex of the projection of $ \Gamma $, the set of all points of $ \Gamma $ that project onto $ u
    $ forms a face of $ \Gamma $.
    This face contains exactly two vertices of $ \Gamma $, which therefore have to share an edge of $ \Gamma $.
    But any edge of $ \Gamma $ is contained in a two-dimensional face of $ \Gamma $ with four vertices, a contradiction
    to Lemma~\ref{lem:three_points}.
\end{proof}

\noindent
Since the polytope $ Q $ has only $ 6 $ facets, its dimension cannot exceed $ 5 $. In the case $ \dim(Q) = 5 $, the
polytope $ Q $ would be a simplex and hence would have only $ 6 $ vertices. Note that any polytope that projects down to
a convex heptagon must have at least $ 7 $ vertices. Thus, we have to consider the remaining two cases $ \dim(Q) = 3 $
and $ \dim(Q) = 4 $.

\makeatletter{}\subsubsection{Three-Dimensional Extensions}

In dimension $ 3 $, there are four combinatorial types of polytopes with $ 6 $ facets and at least $ 7 $ vertices (see \cite{BrittonD73}),
which are illustrated in Fig. \ref{fig:three_dim}.

\makeatletter{}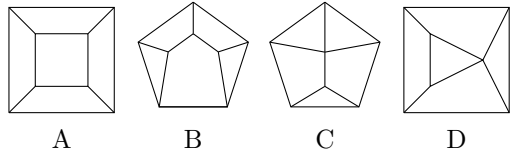
\begin{figure}[H]
    \centering
    \begin{tikzpicture}[scale=0.35]
        \begin{scope}
            \coordinate (w1) at (-2,-2);
            \coordinate (w2) at (-2,2);
            \coordinate (w3) at (2,2);
            \coordinate (w4) at (2,-2);
            \coordinate (w5) at (-1,-1);
            \coordinate (w6) at (-1,1);
            \coordinate (w7) at (1,1);
            \coordinate (w8) at (1,-1);
            \draw (w1)--(w2)--(w3)--(w4)--(w1);
            \draw (w5)--(w6)--(w7)--(w8)--(w5);
            \draw (w1)--(w5);
            \draw (w2)--(w6);
            \draw (w3)--(w7);
            \draw (w4)--(w8);
            \node at (0,-3) {A};
        \end{scope}

        \begin{scope}[shift={(5,0)}]
            \coordinate (w1) at (90:2.2) + (4,0);
            \coordinate (w2) at (360/5+90:2.2) + (4,0);
            \coordinate (w3) at (2*360/5+90:2.2) + (4,0);
            \coordinate (w4) at (3*360/5+90:2.2) + (4,0);
            \coordinate (w5) at (4*360/5+90:2.2) + (4,0);
            \coordinate (w6) at (90:1) + (4,0);
            \coordinate (w7) at (360/5+90:1) + (4,0);
            \coordinate (w8) at (4*360/5+90:1) + (4,0);
            \draw (w1)--(w2)--(w3)--(w4)--(w5)--(w1);
            \draw (w6)--(w7)--(w3)--(w4)--(w8)--(w6);
            \draw (w1)--(w6);
            \draw (w2)--(w7);
            \draw (w5)--(w8);
            \node at (0,-3) {B};
        \end{scope}

        \begin{scope}[shift={(10,0)}]
            \coordinate (w1) at (90:2.2);
            \coordinate (w2) at (360/5+90:2.2);
            \coordinate (w3) at (2*360/5+90:2.2);
            \coordinate (w4) at (3*360/5+90:2.2);
            \coordinate (w5) at (4*360/5+90:2.2);
            \coordinate (w6) at (0, .3);
            \coordinate (w7) at (0,-1);
            \draw (w1)--(w2)--(w3)--(w4)--(w5)--(w1);
            \draw (w6)--(w7);
            \draw (w6)--(w1);
            \draw (w6)--(w2);
            \draw (w6)--(w5);
            \draw (w7)--(w3);
            \draw (w7)--(w4);
            \node at (0,-3) {C};
        \end{scope}

        \begin{scope}[shift={(15,0)}]
            \coordinate (w1) at (-2,-2);
            \coordinate (w2) at (-2,2);
            \coordinate (w3) at (2,2);
            \coordinate (w4) at (2,-2);
            \coordinate (w5) at (-1,-1);
            \coordinate (w6) at (-1,1);
            \coordinate (w7) at (1,0);
            \draw (w1)--(w2)--(w3)--(w4)--(w1);
            \draw (w5)--(w6)--(w7)--(w5);
            \draw (w1)--(w5);
            \draw (w2)--(w6);
            \draw (w3)--(w7);
            \draw (w4)--(w7);
            \node at (0,-3) {D};
        \end{scope}
    \end{tikzpicture}
    \caption{Combinatorial types of three dimensional polytopes with at most $6$ facets and at least $7$ vertices.}
    \label{fig:three_dim}
\end{figure}

\begin{itemize}
    \item If $ Q $ is of type $ A $ or $ B $, then it has $ 8 $ vertices of which exactly one pair of them is projected onto
    the same vertex $ u $ of $ P $.
    Thus, the preimage of $ u $ induces a face of $ Q $ containing exactly two vertices of $ Q $, hence these two
    vertices must share an edge.
    Since any edge of $ Q $ is contained in a $ 2 $-dimensional face of $ Q $ with at least $ 4 $ vertices, this yields
    a contradiction to Lemma \ref{lem:three_points}.

    \item If $ Q $ is of type $ C $, then it has $7$ vertices and thus no two of them are projected on the same point.
    Note that six vertices of $ Q $ form a triangular prism, a contradiction to Lemma \ref{lem:no_prism}.

    \item If $ Q $ is of type $ D $, then due to counting, again no two of its vertices are projected on the same point.
    Label the vertices of $ Q $ by $ w_1,\dotsc,w_7 $ as denoted in Fig.~\ref{fig:three_dim_special} and let $ u_i $
    denote the projection of $ w_i $ for $ i = 1,\dotsc,7 $.

    The lines $ \overline{w_1 w_2} $ and $ \overline{w_3 w_4} $ are not skew since they lie in the same facet of $ Q $.
    By condition~(\ref{cond:parallel}), the lines $ \overline{w_1 w_2} $ and $ \overline{w_3 w_4} $ are not parallel,
    since otherwise the lines $ \overline{u_1 u_2} $ and $ \overline{u_3 u_4} $ are parallel.
    Thus, we have that the lines $ \overline{w_1 w_2} $ and $ \overline{w_3 w_4} $ have a unique
    common point.
    Analogously, one obtains that $ \overline{w_1 w_5} \cap \overline{w_3 w_7} $ and $ \overline{w_2 w_5} \cap
    \overline{w_4 w_6} $ each consists of a single point.

    Moreover, the points $ \overline{w_1 w_2} \cap \overline{w_3 w_4} $, $ \overline{w_1 w_5} \cap \overline{w_3 w_7} $
    and $ \overline{w_2 w_5} \cap \overline{w_4 w_6} $ belong to the intersection of the hyperplane generated by $w_1$,
    $w_2$, $w_5$ and the hyperplane generated by $w_3$, $w_4$, $w_6$, $w_7$;
    i.e. the points $ \overline{w_1 w_2} \cap \overline{w_3 w_4} $, $ \overline{w_1 w_5} \cap \overline{w_3 w_7} $ and $
    \overline{w_2 w_5} \cap \overline{w_4 w_6} $ lie in the same line.
    Hence, $ u_1 $, \ldots, $ u_7 $ violate condition~(\ref{cond:special}).
\end{itemize}

\makeatletter{}\begin{figure}
    \begin{center}
        \begin{tikzpicture}[scale=0.7]
            \coordinate (w1) at (-2,-2);
            \coordinate (w2) at (-2,2);
            \coordinate (w3) at (2,2);
            \coordinate (w4) at (2,-2);
            \coordinate (w5) at (-1,-1);
            \coordinate (w6) at (-1,1);
            \coordinate (w7) at (1,0);
            \node[below, left] at (w1) {$w_4$};
            \node[below, left] at (w2) {$w_3$};
            \node[below, right] at (w3) {$w_7$};
            \node[below, right] at (w4) {$w_6$};
            \node[below, right] at (w7) {$w_5$};
            \node[below, left] at (w5) {$w_2$};
            \node[below, left] at (w6) {$w_1$};
            \draw (w1)--(w2)--(w3)--(w4)--(w1);
            \draw (w5)--(w6)--(w7)--(w5);
            \draw (w1)--(w5);
            \draw (w2)--(w6);
            \draw (w3)--(w7);
            \draw (w4)--(w7);
        \end{tikzpicture}
    \end{center}
    \caption{Treatment of case D.}
    \label{fig:three_dim_special}
\end{figure}
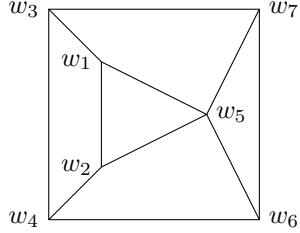

\makeatletter{}\subsubsection{Four-Dimensional Extensions}

In dimension $ 4 $, every polytope with exactly $ 6 $ vertices is the convex hull of a union of a $ 4 $-simplex with one
point. Dualizing this observation yields that every $ 4 $-polytope with exactly $ 6 $ facets is combinatorially
equivalent to a $ 4 $-simplex $ \Delta $ intersected with one closed affine half-space $ H $.
The combinatorial structure of
such an intersection is completely defined by the number $ k $ of the vertices of $ \Delta $ lying on the boundary of $
H $ and the number $ t $ of the vertices of $ \Delta $ lying outside of $ H $. In this case, the total number of
vertices of $ Q $ equals
\[
    (5-t) + (5-k-t) t,
\]
i.e. $ 5 - t $ vertices of $ \Delta $ are vertices of $ Q $ and all other vertices of $ Q $ are intersections of the
edges between the $ t $ vertices of $ \Delta $ lying outside of $ H $ and $ 5-k-t $ vertices of $ \Delta $ lying strictly inside of $ H $. Since the number of vertices of $ Q $ should be at least $ 7 $, there are five possibilities for the
pair $ (k,t) $, namely $ (0,1) $, $ (1,1) $, $ (0,2) $, $ (1,2) $ and $ (0,3) $. In order to finally rule out the
existence of $ Q $, by Lemma \ref{lem:no_prism}, it suffices to show that in each of these cases, $ Q $ contains a
triangular prism as a facet of which at most one pair of vertices is projected onto the same point.

\begin{itemize}
    \item $ (k,t) \in \{ (0,1), \, (1,1), \, (1,2), \, (0,3) \} $: In these cases, the number of vertices of $ Q $ is at
    most $ 8 $ and thus, at most one pair of vertices of $ Q $ is projected onto the same point. There exists a facet $ F $ of
    $ \Delta $ such that none of its vertices lies on the boundary of $ H $ and one or two of its vertices lie
    outside of $ H $. Consider the facet $ F' := F \cap H $ of $ Q $, which is a $ 3 $-simplex intersected with one
    half-space. In particular, by the choice of $ F $, $ F' $ is a triangular prism.

    \item $ (k,t) = (0,2) $: In this case, $ Q $ has exactly $ 9 $ vertices. Let $ w $ be a vertex of $ Q $ such that
    its projection coincides with the projection of another vertex of $ Q $. Let $ F $ be a facet of $ \Delta $ that
    does not contain $ w $. Then  at most one pair of vertices of the facet $ F \cap H $ of $ Q $ is projected onto the same point. To finish the proof, note that for $ (k,t) = (0,2) $ the intersection of every facet of $\Delta$ with the half-space $H$ is a triangular prism.
\end{itemize}

    \makeatletter{}\section{Higher Dimensional Constructions}

We will end our paper by showing that heptagons in general position are not the only polytopes that force minimum size
extensions to have hidden vertices. In fact, Theorem \ref{thm:product} yields families of polytopes in arbitrary
dimensions for which, in any minimum size extension, at least a constant fraction of its vertices are hidden.

\begin{lem}
    \label{lem:product_minimal_extensions}
    Let $ P \subseteq \RR^p $ be a polytope and $ Q \subseteq \RR^{(p+1)+q} $ a minimum size extension of $ P \times [0,1] $.
    Then, the sets
    \begin{align*}
        F_0 &:= \{ (x,y) \in Q : x \in \RR^{p+1}, \, y \in \RR^q, \, x_{p+1} = 0 \} \\
        F_1 &:= \{ (x,y) \in Q : x \in \RR^{p+1}, \, y \in \RR^q, \, x_{p+1} = 1 \}
    \end{align*}
    are both minimum size extensions of $ P $.
\end{lem}
\begin{proof}
    Note that $ F_0 $ and $ F_1 $ are both extensions of $ P $ and proper faces of $ Q $.
    Let $ k $ denote the number of facets of $ Q $ and $ t $ be the extension complexity of $ P $.
    Clearly, we have that $ k \leq t + 2 $ holds.
    For $ i \in \{0,1\} $ let $ f_i \geq t $ be the number of facets of $ F_i $ and let us define the sets
    \begin{align*}
        \mathcal{C}_i & := \{ F : F \text{ facet of } Q, \, F_i \subseteq F \} \\
        \mathcal{D}_i & := \{ F : F \text{ facet of } Q, \, F_i \cap F = \emptyset \}.
    \end{align*}
    It is straightforward to see that
    \begin{equation}
        \label{ineq:product_bound_number_of_facets}
        f_i \leq k - |\mathcal{C}_i| - |\mathcal{D}_i| \leq (t + 2) - |\mathcal{C}_i| - |\mathcal{D}_i|
    \end{equation}
    holds.
    Thus, it remains to show that $ |\mathcal{C}_i| + |\mathcal{D}_i| \geq 2 $ holds for $ i=0,1 $.
    Clearly, this inequality holds if neither $ F_0 $ is a facet nor is $ F_1 $, since this implies $ |\mathcal{C}_i|
    \geq 2 $ for $ i=0,1 $.

    Towards this end, by symmetry, it remains to consider the case that $ F_0 $ is a facet of $ Q $. Since $ |\mathcal{C}_i| \geq 1 $ for $ i=0,1 $, it is enough to show that in this case we have $ |\mathcal{D}_i| \geq 1 $ for $ i=0,1 $.
    Indeed, since $ F_0 $ and $ F_1 $ are disjoint, we obtain $ F_0 \in \mathcal{D}_1 $ and thus $ |\mathcal{D}_1| \geq 1 $.
    Due to $ t \leq f_1 $ and inequality \eqref{ineq:product_bound_number_of_facets},
    it holds that $ |\mathcal{C}_1| + |\mathcal{D}_1| \leq 2 $ and hence $ |\mathcal{C}_1| = 1 $.
    Thus, $ F_1 $ has to be a facet, too. Moreover, the facet $ F_1 $ is in $ \mathcal{D}_0 $, and thus $ |\mathcal{D}_0| \geq 1 $.
\end{proof}

\begin{thm}
    \label{thm:product}
    Let $ P $ be a convex heptagon in general position and $ Q $ a minimum size extension for $ P \times [0,1]^d $. Then, for
    at least $ \frac{1}{9} $ of the vertices of $ Q $, we have that none of them is projected onto a vertex of $ P
    \times [0,1]^d $.
\end{thm}
\begin{proof}
    We will prove the statement by induction over $ d \geq 0 $. In the case of $ d = 0 $, by Theorem
    \ref{thm:additional_vertices_heptagon} and its proof, we know that $ Q $ has at most $ 9 $ vertices and that at
    least one of them is not projected onto a vertex of $ P $.

    For $ d \geq 1 $, let $ P' := P \times [0,1]^{d-1} \subseteq \RR^{d+1} $ and $ Q $ be a minimum size extension of $ P
    \times [0,1]^d = P' \times [0,1] \subseteq \RR^{d+2} $. Observe that the vertex set of $P' \times [0,1]$ is the  cartesian product of the vertex set of $P'$ and the set $\{0,1\}$. Let us partition the set $ V $ of vertices of $ Q $ into the
    three sets
    \begin{align*}
        V_0 & := \{ (x,y) \in V : x \in \RR^{d+2}, \, y \in \RR^q, \, x_{d+2} = 0 \} \\
        V_1 & := \{ (x,y) \in V : x \in \RR^{d+2}, \, y \in \RR^q, \, x_{d+2} = 1 \} \\
        V_{\star} & := \{ (x,y) \in V : x \in \RR^{d+2}, \, y \in \RR^q, \, 0 < x_{d+2} < 1 \}
    \end{align*}
    Clearly, no vertex in $ V_{\star} $ is projected onto a vertex of $ P' \times [0,1] $. By Lemma
    \ref{lem:product_minimal_extensions}, $ \conv(V_0) $ and $ \conv(V_1) $ are minimum size extensions of $ P' $. By
    induction, for at least $ \frac{1}{9} (|V_0| + |V_1|) $ of $ V_0 \cup V_1 $, we have that none of them is projected
    onto a vertex of $ P' \times [0,1] $. Thus, the number of vertices that are not projected onto a vertex is at
    least $ \frac{1}{9} (|V_0| + |V_1|) + |V_{\star}| \geq \frac{1}{9} |V| $.
\end{proof}
 
    \makeatletter{}\section*{Acknowledgements}
\noindent
We would like to thank Samuel Fiorini for valuable comments on this work.

        \bibliographystyle{plain}

\end{document}